\documentclass[11pt, reqno]{amsart}

\usepackage{amsfonts}
\usepackage{amssymb}
\usepackage{amsthm}
\usepackage{amsmath}
\usepackage{enumerate}

\usepackage{graphicx}

\newtheorem{thm}{Theorem}[section]
\newtheorem{cor}[thm]{Corollary}

\newtheorem{lemma}[thm]{Lemma}

\newtheorem{rmk}[thm]{Remark}

\newtheorem{prop}[thm]{Proposition}
\newtheorem{conj}[thm]{Conjecture}
\newtheorem{question}[thm]{Question}

\numberwithin{equation}{section}
\begin{document}
\title{Minimal equatorial fillings}
\author{Jacob Bernstein}\address{Johns Hopkins University \\ 3400 N. Charles St, Baltimore MD 21218 USA}
	\email{jberns15@jhu.edu}
\author{Daniel Ketover}\address{Rutgers University\\  Busch Campus - Hill Center \\ 110 Freylinghausen Road, Piscataway NJ 08854 USA} \email{dk927@math.rutgers.edu}

\begin{abstract}
We show that a great circle in $\mathbb{S}^3$ bounds embedded non-orientable minimal surfaces of unbounded genera and all possible non-zero Euler numbers.  This answers a question of R. Hardt and H. Rosenberg (1990). It also gives the first example of a real analytic Jordan curve bounding infinitely many embedded minimal surfaces of distinct topological types.  Such surfaces may arise as the links of boundary singularities of non-orientable minimal hypersurfaces in $\mathbb{R}^4$.  
\end{abstract}
\maketitle
\section{Introduction}

By the monotonicity formula, the singularities of minimal hypersurfaces are modeled on cones whose links are minimal hypersurfaces in the round sphere. In the same way, boundary singularities of minimal hypersurfaces are modeled on cones over minimal hypersurfaces with equatorial boundary in the round sphere.

In 1979, R. Hardt and L. Simon (Theorem 11.1 in \cite{HS}) obtained the following result:
\begin{thm}[Hardt-Simon (1979)]\label{orientable}
The equatorial $\mathbb{S}^{n-2}\subset\mathbb{S}^n$ bounds no embedded orientable minimal hypersurface aside from minimal hemispheres.
\end{thm}

Using Theorem \ref{orientable}, Hardt-Simon were able to rule out boundary singularities for solutions to the oriented Plateau problem.  

Given a codimension-two \emph{non-orientable} submanifold of Euclidean space, however, there may be no smooth solution to the corresponding unoriented Plateau problem for topological reasons. For instance, the manifold $\mathbb{RP}^2$ is not the boundary of any smooth three-manifold. As one may embedded $\mathbb{RP}^2$ into $\mathbb{R}^4$, it follows that there cannot be any regular solution to the Plateau problem (among mod $2$ flat chains) bounded by the image.  Since interior regularity of area minimizers has been established (W. Fleming \cite{Fleming2}, F. Almgren, R. Schoen and L. Simon \cite{ASS}) this implies that a putative minimizer must have a singularity at its boundary $\mathbb{RP}^2$.  As the links of such singularities are embedded minimal hypersurfaces in $\mathbb{S}^3$ with equatorial boundary it is a natural problem to understand and classify these objects.   

In 1970, B. Lawson \cite{LS} discovered an embedded minimal M\"obius band in $\mathbb{S}^3$ with boundary a great circle.  White \cite{LectureNotesOtis} proposed the cone over this M\"{o}bius band as an explicit example of a non-orientable boundary singularity of an area-minimizer\footnote{During the final preparation of this manuscript Guaraco-Parise \cite{guaracoWhitesConeMobius2026} showed the cone over the Lawson band (defined in \eqref{signs}) is area-minimizing.}. R. Hardt and H. Rosenberg \cite{HR} asked the following:
\begin{question}[R. Hardt and H. Rosenberg (1990)]
Is the Lawson M\"obius band the unique embedded non-orientable minimal surface in $\mathbb{S}^3$ with boundary a great circle?
\end{question}

In this paper, we answer their question negatively by constructing a two-parameter family $\tilde{\xi}_{k.m}$ of such surfaces.  In several respects the surfaces $\tilde{\xi}_{k.m}$ are non-orientable and twisted analogs to the closed orientable Lawson surfaces $\xi_{k,m}\subset\mathbb{S}^3$ found in 1970 \cite{LS}. 

If $\Sigma\subset\mathbb{S}^3$ is an embedded surface bounded by the great circle $C$, its \emph{Euler number}\footnote{The reason for the terminology is that if we cap $\Sigma$ off by a disk in $B^4$, this integer coincides with the Euler number associated to the (twisted) Euler class of the resulting closed surface's normal bundle in $\mathbb{R}^4$ (see the discussion in Section 2 in \cite{Conway}). In the knot theory literature, this integer is sometimes called the \emph{boundary slope}.}  $e(\Sigma)$ is the even integer $k$ such that, if $T(C)$ is a small tubular neighborhood of $C$, then $\Sigma \cap \partial T(C)$ is a $(1,k)$ torus knot in $\partial T(C)$.  When $\Sigma$ is orientable, $e(\Sigma)=0$.

The following is our main result:

\begin{thm}\label{main}
Fix $k\in\mathbb{N}$.  For each odd $m\in\mathbb{N}$ at least $3$, there exists an embedded non-orientable minimal surface $\tilde{\xi}_{k,m}\subset\mathbb{S}^3$ so that the following hold:
\begin{enumerate}[(i).]
\item $\partial\tilde{\xi}_{k,m}$ is a great circle. \label{item1}
\item $\mbox{genus}(\tilde{\xi}_{k,m})= mk+1$ and $e(\tilde{\xi}_{ k,m})=2(k+1)$.\label{genus}
\item $\mbox{Area}(\tilde{\xi}_{k,m})< \mbox{Area}(\overline{\tau}_{1,2(k+1)})$. \label{item3}
\item The symmetry group of $\tilde{\xi}_{k,m}$ is $\mathbb{D}_{m(k+1)}\subset SO(4)$. \label{item4}
\item $\tilde{\xi}_{k,m}$ contains $m(k+1)$ half-geodesics meeting $\partial\tilde{\xi}_{k,m}$  orthogonally in $2m(k+1)$ equally spaced points. \label{item5}
\item For any sequence $m_i\to\infty$ there holds \begin{equation}\lim_{i\to\infty} \tilde{\xi}_{k,m_i}= \overline{\tau}_{1,2(k+1)}\mbox{ in the sense of varifolds}.\end{equation} \label{item6}
\end{enumerate}
\end{thm}

The surfaces $\overline{\tau}_{1,2(k+1)}$ referenced in item (\ref{item6}) are halves of Lawson's immersed Klein bottles discussed in the next section.

Specializing Theorem \ref{main} to $k=1$ we find that all even genera at least four are realized as minimal embeddings:
\begin{cor}
For each even $m\geq 4$, the embedded non-orientable minimal surface $\tilde{\xi}_{1,m-1}$ has genus $m$ and Euler number $4$.
\end{cor}

We also get that nearly all genera are attained:
\begin{cor}
If $g\neq 2^N+1$ for all $N\geq 0$, then there exists an embedded non-orientable minimal surface in $\mathbb{S}^3$ of genus $g$ and boundary a great circle. 
\end{cor}

To the authors' knowledge, Theorem \ref{main} gives the first example of a real analytic Jordan curve bounding infinitely many embedded minimal surfaces of distinct topological types.  In 1950, R. Courant \cite{C} described the existence of a Jordan curve with one singular point bounding infinitely many minimal surfaces (later made rigorous by W. Fleming \cite{Fleming}).\footnote{Morgan (\cite{M1} (1980)) also found a \emph{disconnected} curve in $\mathbb{R}^3$ with four components bounding minimal surfaces of arbitrarily high genera.  See also J. Pitts and H. Rubinstein (\cite{PR} (1987)) for an example in $\mathbb{R}^3$ with three components.  A sketch going back to J. Pitts suggests two orthogonal linked circles in $\mathbb{R}^3$ might also bound infinitely many minimal surfaces.} In 1976, F. Morgan \cite{M2} found a real analytic closed curve in $\mathbb{R}^4$ bounding a continuum of minimal hypersurfaces of the same topological type.  Recently C. De Lellis, G. De Philippis and J. Hirsch \cite{DDH} found an example of a smooth curve in dimension four bounding a minimal surface of infinite genus. All of these works consider area-minimizing surfaces, while the surfaces in Theorem \ref{main} are unstable (but equivariantly minimizing with respect to suitable group actions).

The surfaces $\tilde{\xi}_{k,m}$ are analogous to the orientable (closed) Lawson surfaces $\xi_{k,m}$ in several ways (which is the reason for the notation).  The latter converge in the sense of varifolds as $m\to\infty$ to a stationary integral varifold $V_k$ consisting of $k+1$ equally spaced great spheres intersecting along a great circle with density $k+1$. Similarly, the limit $W_k$ of $\tilde{\xi}_{k,m}$ as $m\to\infty$ is smooth away from a great circle where it has density $k+1$.  In a tubular neighborhood $T$ of this circle, $V_k$ is union of Seifert fibers\footnote{Recall that if we consider $D^2\times I$ fibered by the segments $\{x\}\times I$, and we identify the top and bottom faces with a twist angle of $2\pi q/p$ then we obtain a fibered solid torus with Seifert fiber data $(p,q)$. The core curve $\{0\}\times I$ in the quotient solid torus is called the \emph{singular} fiber, and the other \emph{regular} fibers wrap around it.} of $T$ with fibering data $(1,0)$.    The varifold $W_k$ restricted to $T$, on the other hand, is also union of Seifert fibers but instead with fibering data $(2(k+1),1)$.  It is for this reason that we call the surfaces $\tilde{\xi}_{k,m}$ ``twisted."  While $\tilde{\xi}_{k,m}$ only comprises half of a closed minimal surface, the area of $\tilde{\xi}_{k,m}$ approaches that of $k+1$ great spheres when $k$ and $m$ are sufficiently large, which is the same behavior as the Lawson surfaces.  As $m\to\infty$, the Lawson surfaces  are modeled by Scherk's singly-periodic minimal surfaces found in 1834, while the $\tilde{\xi}_{k,m}$ are modeled (near the polar geodesic to their boundary) by the twisted Scherk surfaces discovered by H. Karcher \cite{Karcher} in 1988.  

It is natural to ask what possible combinations of Euler numbers and genera can be realized by minimal embeddings with boundary an equator.  By results of H. Whitney \cite{Whitney} and W. Massey \cite{Ma} (collected in Theorem \ref{admissible}), not all pairs of Euler numbers and genera can even be realized by \emph{smooth embeddings} in $\mathbb{S}^3$ with boundary an unknot.  In fact, the surfaces obtained in Theorem \ref{main} do not realize the infimal possible genus\footnote{We exhibit these genus-minimizing configurations in Theorem \ref{admissible}.} given their Euler number. For example, Euler number $4$ permits a smooth embedding of a genus $2$ surface (i.e. a punctured Klein bottle) but Theorem \ref{main} produces a minimal embedding $\hat{\xi}_{1,3}$ with smallest possible genus $4$ (i.e. a punctured Klein bottle with a handle attached). 

To obtain infimal genus surfaces given their Euler numbers (and also to realize the few genera missed in Theorem \ref{main}) one would need to extend Theorem \ref{main} to the case $m=1$.  In this borderline case, however, equivariant minimization under the corresponding symmetry group appears to produce a hemisphere with $k+1$ crosscaps collapsing at equally distributed points on its boundary.

Let us describe the construction of the surfaces $\tilde{\xi}_{k,m}$. Lawson \cite{LS} found two families of minimal surfaces in $\mathbb{S}^3$ by solving the Plateau problem for a four sided geodesic polygon and reflecting to obtain a closed surface.   Later, H. Karcher, U. Pinkall and I. Sterling \cite{KPS} found closed embedded minimal surfaces in $\mathbb{S}^3$ associated to the Platonic solids by first finding a free boundary minimal surface in a fundamental domain and then reflecting.  Our method combines both approaches. To construct the surfaces $\tilde{\xi}_{k,m}$ we solve the Plateau problem for a six sided geodesic polygon and apply successive Schwarz reflections to obtain a surface with boundary a great circle. While doing this, we need to impose an additional involutive symmetry on each patch.  A result of Meeks-Yau \cite{MYD} allows one to solve such an equivariant Plateau problem.  We note that six sided geodesic polygons were considered by J. Choe and M. Soret \cite{CS2} in producing a desingularization of several Clifford tori in $\mathbb{S}^3$ meeting at equal angles.  The surfaces $\tilde{\xi}_{k,m}$, like those of Choe-Soret, have no reflective symmetries.

In fact, the existence of non-orientable minimal surfaces with boundary a great circle follows from topological considerations in lens spaces.  If $p$ is odd, it follows from the Universal Coefficient Theorem that the group $H_1(L(p,q),\mathbb{Z}_2)$ is trivial while $H_1(L(p,q),\mathbb{Z})\cong\mathbb{Z}_p$.  Thus if we project a great circle in $\mathbb{S}^3$ to get a closed embedded curve $\gamma\subset L(p,q)$ generating $\pi_1(L(p,q))$ and solve for a mod $2$ flat chain with infimal area and boundary $\gamma$, we necessarily obtain a non-orientable minimal surface. This surface lifts to a smooth embedded minimal surface $\mu_{p,q}$ (though its genus is not specified).  In Section \ref{lenssection}, however, we show that the $\tilde{\xi}_{k,m}$ surfaces for appropriate $m$ and $k$ (together with their reflections through great spheres) indeed descend to the lens spaces $L(p,q)$ and thus are likely candidates for the minimal surfaces obtained from geometric measure theory. 

There is an analogous and classical situation for lens spaces $L(p,q)$ with even $p$.  Indeed, if $p=2k$, since $H_2(L(2k,q),\mathbb{Z}_2)\cong\mathbb{Z}_2$ but $H_2(L(2k,q),\mathbb{Z})\cong 0$  there exist homologically non-trivial non-orientable surfaces in $L(2k,q)$.  In 1969, G.E. Bredon and J.W. Wood \cite{Bre} determined the infimal genus of such a surface that can be realized by an embedding in terms of data from the iterated fraction expansion of the ratio $p/q$.

Finally we show in Section \ref{symmetrysection} that only cyclic and dihedral groups can occur as symmetry groups for minimal surfaces with boundary a great circle.  This raises the question of whether the constructions in this paper exhaust all examples. 

The organization of this paper is as follows.  In Section \ref{immersecsection} we introduce the immersed Lawson M\"obius bands. In Section \ref{mainsection} we prove Theorem \ref{main}.  In Section \ref{lenssection} we consider finite free actions on $\mathbb{S}^3$ and the resulting minimal surfaces we obtain by lifting surfaces in lens spaces.  In Section \ref{admissiblesection} we discuss the restrictions on Euler number and genus for non-orientable embedded surfaces with boundary an unknot due to Whitney and Massey. In Section \ref{symmetrysection}  we consider symmetry of minimal surfaces with boundary a great circle.  In Section \ref{questionsection} we collect questions related to constructions in this paper. 
\newline
\newline
\emph{Acknowledgements:} J.B was partially supported by the NSF grant DMS-2203132.  D.K. was partially supported by the NSF grant DMS-2405114.

\section{The immersed Lawson M\"obius bands}\label{immersecsection}
In 1970, Lawson discovered a family of Klein bottles and tori $\tau_{m,k}$ minimally immersed in $\mathbb{S}^3$ given by the formula
\begin{equation}
F_{m,k}(x,y):\mathbb{R}^2\to\mathbb{S}^3\subset\mathbb{R}^4
\end{equation}
\begin{equation}\label{signs}
F_{m,k}(x,y) = (\cos(m x)\sin(y),\sin(m x)\sin(y),\cos(kx)\cos(y), \sin(kx)\cos(y)),
\end{equation}
Here $k$ and $m$ are relatively prime integers.  If $m$ and $k$ are both odd, then $\tau_{m,k}$ is an immersed minimal torus.  If either $m$ or $k$ is even, then $\tau_{m,k}$ is an immersed minimal Klein bottle.  

If $k$ is even, the image of the map $F_{1,k}$ restricted to $\mathcal{D}:=[0,\pi]\times [-\pi/2,\pi/2]$ spans an immersed M\"obius band that we denote $\overline{\tau}_{1,k}$. The boundary of $F_{1,k}$ is parameterized at $y=\pm\frac{\pi}{2}$. 
Writing
\begin{equation}
\mathbb{S}^3=\{(z,w)\in\mathbb{C}^2\;|\; |z|^2+|w|^2=1\}, 
\end{equation}
we can express the map \eqref{signs}
\begin{equation}\label{param}
F_{1,k}(x,y)= (e^{ix}\sin y,e^{ikx}\cos y).
\end{equation}
Note that $\partial\overline{\tau}_{1,k}=C$, where $C$ denotes the great circle 
\begin{equation}C:=\{(z,0)\in\mathbb{C}^2\;|\; |z|=1\}.\end{equation}
If $k=2$, then $\overline{\tau}_{1,k}$ is embedded.  If $k>2$ then the surface $\overline{\tau}_{1,k}$ is embedded away from the polar circle $C^*$ to $C$:
\begin{equation}C^*:=\{(0,w)\in\mathbb{C}^2\;|\; |w|=1\}.\end{equation}
Along the great circle $C^*$, there are $k/2$ sheets of $\overline{\tau}_{1,k}$ that meet at equal (consecutive) angles $\frac{2\pi}{k}$.  Indeed, let us verify this.  If $y=0$ then 
\begin{equation}
F_{1,k}(x,0) = (0,e^{ikx}).
\end{equation}
Given a point $(0,e^{i\alpha})\in C^*$, there are $k$ preimages $F_{1,k}^{-1}((0,e^{i\alpha}))$ on the segment $\{y=0\}$ in $\mathcal{D}$ given by  $\alpha_j =\alpha/k + 2\pi j/k$ ($j=0,1,..,k-1$).
Moreover, we can compute
\begin{equation}
D_j=\frac{\partial}{\partial y}\Big|_{y=0}F_{1,k}(\alpha_j,y)=(e^{i\alpha_j},0).
\end{equation}
and 
\begin{equation}
\frac{\partial}{\partial x}\Big|_{y=0}F_{1,k}(\alpha_j,y)=(0,ike^{ik\alpha_j}), 
\end{equation}
which points along $C^*$.  Thus we get
\begin{equation}
( D_j, D_{j+1}) =(\cos( \alpha_j), \sin(\alpha_j),0,0)\cdot (\cos(\alpha_{j+1}),\sin(\alpha_{j+1},0,0) = \cos(\alpha_{j+1}-\alpha_j).
\end{equation}

Since $\alpha_{j+1}-\alpha_j=\frac{2\pi}{k}$, it follows that the angle $\theta$ at which successive sheets of $\overline{\tau}_{1,k}$ meet at $(0,e^{i\alpha})\in C^*$ is given by 
\begin{equation}
\theta = \arccos((D_j, D_{j+1})) = \arccos (\cos(\frac{2\pi}{k})) = \frac{2\pi}{k}.
\end{equation}

The Lawson bands $\overline{\tau}_{1,k}$ are also ruled as the segments $\{x=\mbox{const}\}$ correspond to geodesic segments of length $\pi$.  By a result of Lawson (\cite{LS}, Proposition 7.2), the bands $\overline{\tau}_{1,k}$ are the unique ruled minimal surfaces with boundary a great circle in $\mathbb{S}^3$.
\subsection{Euler number}
Differentiating \eqref{param} with respect to $\eta$ at $y=\pm\pi/2$ we get the inward conormals along $C$:
\begin{equation}\label{upperderiv}
\nu_{\overline{\tau}_{1,k}}(x,\frac{\pi}{2}) = -\frac{\partial}{\partial y}\Big|_{y=\pi/2} F_{1,k}(x,y) =  (0,  e^{ikx}), 
\end{equation}
and 
\begin{equation}\label{lowerderiv}
\nu_{\overline{\tau}_{1,k}}(x,-\frac{\pi}{2}) = \frac{\partial}{\partial y}\Big|_{y=-\pi/2} F_{1,k}(x,y) =  (0,  e^{ikx}).
\end{equation}
Fix a hemisphere in $\mathbb{S}^3$ with boundary $C$ with inner conormal (constant along $C$) $\nu_0=(0,0,1,0)$. Note that \eqref{upperderiv} and \eqref{lowerderiv} imply that 
\begin{equation}\label{conormaldot}
\nu_0(x,\pm\frac{\pi}2{})\cdot \nu_{\overline{\tau}_{1,k}}(x,\pm\frac{\pi}{2})=(0,0,1,0)\cdot (0,0,\cos(kx),\sin(kx))=\cos(kx).  
\end{equation}
We obtain from \eqref{conormaldot} and the boundary identifications of the top and bottom of the domain $\mathcal{D}=[0,\pi]\times [-\frac{\pi}{2},\frac{\pi}{2}]$ that
\begin{equation}
e(\overline{\tau}_{1,k})= k.
\end{equation}
\subsection{Symmetries}\label{symmetries}
Let us also determine the symmetries in $SO(4)$ of the Lawson bands $\overline{\tau}_{1,k}$.  In fact, they are real algebraic surfaces of degree $k+1$:
\begin{equation}\label{algform}
\overline{\tau}_{1,k}=\{(z,w)\in\mathbb{S}^3\;|\;\mbox{Im}(z^{k}\overline{w})=0\mbox{ and } \mbox{Re}(z^{k}\overline{w})\geq0\}.
\end{equation}

From \eqref{algform} it is apparent that the symmetry group of $\overline{\tau}_{1,k}$ is the group $\mathbb{S}^1\rtimes\mathbb{Z}_2$ where the $\mathbb{S}^1$-action is given by 
\begin{equation}
e^{i\theta}(z,w) = (e^{i\theta}z,e^{ik\theta}w).  
\end{equation}
The $\mathbb{Z}_2$-action is conjugation:
\begin{equation}
c(z,w) = (\overline{z},\overline{w}).
\end{equation}
Note that in coordinates $(x_1,x_2,x_3,x_4)$ of $\mathbb{R}^4$ we have
\begin{equation}
c(x_1,x_2,x_3,x_4) = (x_1,-x_2,x_3,-x_4), 
\end{equation}
which implies that $\mbox{det}(c)=1$.  Thus $c\in SO(4)$.

Let $c_2:\mathbb{S}^3\to\mathbb{S}^3$ be given by 
\begin{equation}
c_2(z,w) = (z,\overline{w}), 
\end{equation}
or equivalently
\begin{equation}
c_2(x_1,x_2,x_3,x_4)=(x_1,x_2,x_3,-x_4), 
\end{equation}
Note from \eqref{signs} that
\begin{equation}\label{secondconj}
c_2(\overline{\tau}_{1,k})=\overline{\tau}_{1,-k}.
\end{equation}
\subsection{Area}
The area of $\overline{\tau}_{1,k}$ can be expressed by
\begin{equation}
\mbox{Area}(\overline{\tau}_{1,k})=2\pi kE(1-\frac{1}{k^2}), 
\end{equation}
where $E(p)$ denotes the elliptic integral of the second kind 
\begin{equation}
E(p) =\int_0^{\pi/2}\sqrt{1-p\sin^2(\theta)}d\theta.
\end{equation}
For $\varepsilon$ near zero, there holds (\cite{DLMF}, Section 19.12.2):
\begin{equation}
E(1-\varepsilon)
=1-\frac{\varepsilon}{4}\,\ln\\\varepsilon
+
O(\varepsilon),
\end{equation}
and thus we may expand
\begin{equation}
|\overline{\tau}_{1,2k}|=4\pi k E(1-\frac{1}{4k^2})= 4\pi k +\frac{\pi\ln(k)}{2k}+O(\frac{1}{k}).
\end{equation}
In particular, 
\begin{equation}\label{limitarea}
\lim_{k\to\infty}(|\overline{\tau}_{1,2k}|-4\pi k)=0.
\end{equation}

\section{Proof of Theorem \ref{main}}\label{mainsection}
In this section, we prove Theorem \ref{main}.
\subsection{Set up}
Let us parameterize $C$ by
\begin{equation}
C(\theta)= (e^{i\theta},0)\mbox{ with } 0\leq\theta\leq 2\pi.
\end{equation}
and $C^*$ by
\begin{equation}
C^*(\theta)= (0,e^{i\theta})\mbox{ with } 0\leq\theta\leq 2\pi.
\end{equation}
Note that for all $\theta$ and $\phi$, 
\begin{equation}\label{dist}
\mbox{dist}_{\mathbb{S}^3}(C(\theta),C^*(\phi))=\frac{\pi}{2}.
\end{equation} 
Moreover, the unique minimizing geodesic realizing the infimal distance in \eqref{dist} is given by 

\begin{equation}
\gamma_{\theta,\phi}(t)=(e^{i\theta}\sin t,e^{i\phi}\cos t)\mbox{ for } 0\leq t\leq\frac{\pi}{2}.
\end{equation}
The $\mathbb{S}^1$-family of great hemispheres $\{H_\phi\}_{\phi\in [0,2\pi]}$ in $\mathbb{S}^3$ with boundary $C$ is given by
\begin{equation}
H_\phi =
\{ (z,w) \in S^3 :
\operatorname{Im}(e^{-i\phi} w)=0,\;
\operatorname{Re}(e^{-i\phi} w)\ge 0 \}.
\end{equation}
Note that for each $\phi\in [0,2\pi]$, 
\begin{equation}
H_\phi\cup H_{\pi+\phi}\mbox{ is a great sphere.}
\end{equation}
Note also that \begin{equation}\gamma_{\theta,\phi}\subset H_\phi\mbox{ for all } \theta\in [0,2\pi]\mbox{ and } \phi\in [0,2\pi].\end{equation}
Moreover, the inner conormal $\nu_\phi(\theta)$ to $H_\phi$ along $C$ at $C(\theta)$ is given by 
\begin{equation}
\nu_\phi(\theta) = \dot{\gamma}_{\theta, \phi}(\pi/2)=(0, e^{i\phi}).   
\end{equation}
\noindent
For $i=0,1,...,m-1$ denote equally spaced hemispheres containing $C$:
\begin{equation}
H^i:=H_{\frac{2\pi i}{m}}.
\end{equation}
\noindent
For each $i=0,...,m-1$, let us denote by $W_i$ the closure of the component of \begin{equation}\mathbb{S}^3\setminus (H^i\cup H^{i+1})\end{equation} of smaller volume.  Each convex set $W_i$ is a lens-like three-ball with boundary two hemispheres meeting at angle $\frac{2\pi}{m}$ along $C$.
\noindent
Let $\rho_{\theta}:\mathbb{S}^3\to\mathbb{S}^3$ denote the rotation in $SO(4)$
\begin{equation}
\rho_\theta(z,w)= (e^{i\theta}z,w). 
\end{equation}
and let $\sigma_{\theta}:\mathbb{S}^3\to\mathbb{S}^3$ denote the rotation in $SO(4)$
\begin{equation}
\sigma_\theta(z,w)= (z,e^{i\theta}w). 
\end{equation}

\noindent

\subsection{Fundamental polygon and patch}
We now construct a closed path $\Gamma$ in $\partial W_0$ comprised of six geodesic segments.  For $P,Q\in\mathbb{S}^3$ with \begin{equation}\mbox{dist}_{\mathbb{S}^3}(P,Q)<\pi,\end{equation} we denote by $PQ$ the unique geodesic segment from $P$ to $Q$.  Similarly if $\{P_1,...,P_q\}\in\mathbb{S}^3$ with 
\begin{equation}\mbox{dist}_{\mathbb{S}^3}(P_i,P_{i+1})<\pi\mbox{ for each } i=0,...,q-1,\end{equation} 
we denote by $P_1P_2...P_q$ the piecewise smooth geodesic segment obtained by concatenating the geodesic segments $P_1P_2$, $P_2P_3$,...,$P_{q-1}P_q$.

Let us divide the great circle $C$ into $2m(k+1)$ equally spaced points $\{A_0,...,A_{2m(k+1)-1}\}$ by setting 
\begin{equation}
A_j=(e^{2\pi i\frac{j}{2m(k+1)}},0).
\end{equation}
Similarly, let us divide the great circle $C^*$ into $m$ equally spaced points $\{B_0,...,B_{m-1}\}$ by setting
\begin{equation}
B_j = (0,e^{2\pi i\frac{ j}{m}}).
\end{equation}  
Then define the geodesic circuit comprised of six segments:
\begin{equation}
\Gamma=B_0A_0A_1 B_1 A_{1+m} A_{m}B_0.
\end{equation}
Note that  \begin{equation}\label{pointrotate}\rho_{\frac{2 \pi p}{k+1}}(A_j)=A_{j+2pm}\mbox{ and } \rho_{\frac{2 \pi p}{k+1}}(B_j)=B_j,\end{equation} where the indices in the first equality are understood mod $2m(k+1)$.
It follows that
\begin{equation}\label{howtau}
\partial(\overline{\tau}_{1,2(k+1)}\cap\mbox{int}(W_0)) = \bigcup_{p=0}^{k} \rho_{\frac{2\pi p}{k+1}}( \Gamma).
\end{equation}

    \begin{figure}
	\centering
	\resizebox{3in}{!}{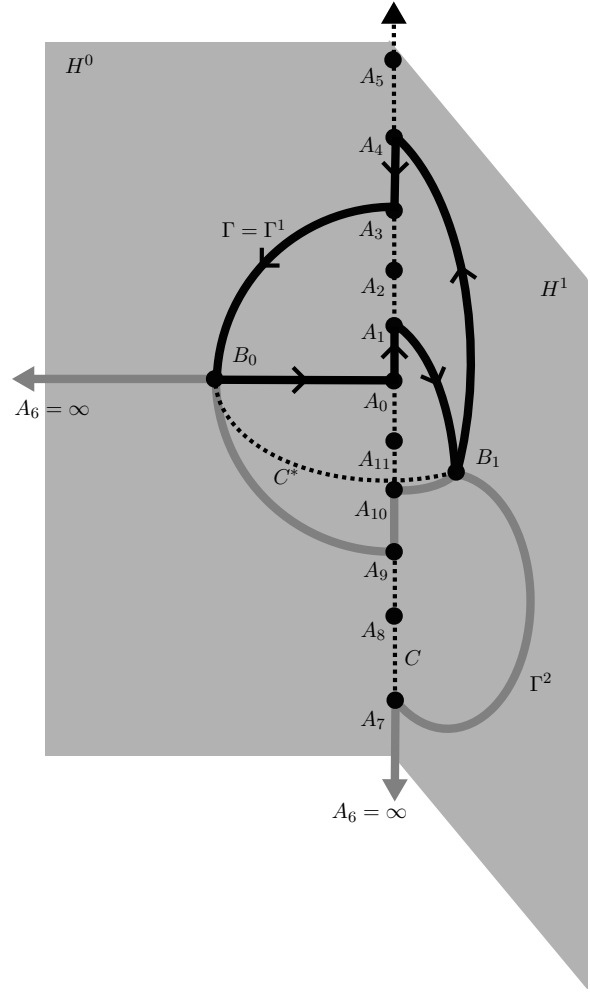}
	\caption{Stereographic projection of $\Gamma=\Gamma^1$ and $\Gamma^2=\rho_{\pi}(\Gamma) $  for $m=3$ and $k=1$.}
	\label{GammaFig}
\end{figure}

   \begin{figure}
	\centering
	\resizebox{3in}{!}{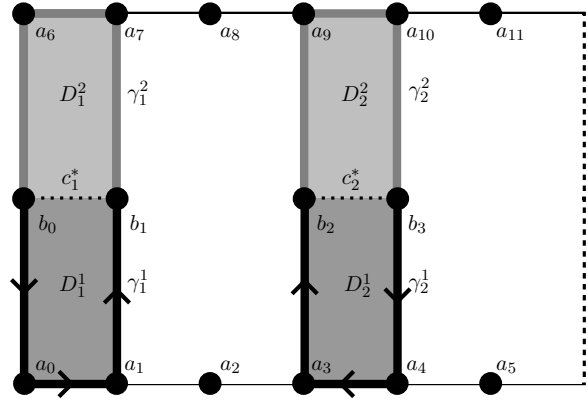}
	\caption[Domain of$F_{1,4}$]{Domain of ${F}_{1,4}$ with following indicated: $F_{1,4}(a_i)=A_i$, $F_{1,4}(b_i)=B_i$ ,  $i$ mod $2$;
	 $F_{1,4}(\gamma_1^i\cup \gamma_2^i)=\Gamma^i$; 
	 ${F}_{1,4}(c_i^*)=C^*\cap W_{0}$;  
	 $F_{1,4}(D_1^i\cup D_2^i)=S_1\cup S_2=\bar{\tau}_{1,4}\cap \overline{W}_0^i$.
	 }
	\label{GammaParamFig}
\end{figure}

To see \eqref{howtau}, recall first that $\overline{\tau}_{1,2(k+1)}$ is ruled by geodesic segments emanating orthogonally from $C$.   As $\theta$ increases from $0$ to $2\pi$, the conormal to $\tau_{1,2(k+1)}$ at $C(\theta)$ points into each of the regions $W_0,...,W_{m-1}$ for an equal fraction ($2\pi/m$) of $\theta$ values. In addition, the range of $\theta$ values for which the conormal points into $W_0$ is itself divided into $2(k+1)$ intervals of equal length that are equally separated.  Thus each such interval has length $\frac{\pi}{m(k+1)}$.  Moreover, when a geodesic on $\overline{\tau}_{1,2(k+1)}$ emanating from $C(\theta)$ has initial velocity vector pointing into $\partial W_0$, it must remain on $\partial W_0$ until the antipodal point $C(\theta+\pi)$ since $\partial W_0$ is totally geodesic.  

Let us define an involution $\iota\in SO(4)$ by
\begin{equation}\label{defiota}
\iota(z, w) = \left( e^{i \frac{\pi(m+1)}{m(k+1)}} \bar{z}, \ e^{i \frac{2\pi}{m}} \bar{w} \right).
\end{equation}
The map $\iota$ is a rotation of $\pi$ about a geodesic that bisects the lens $W_0$, flipping it upside down. It satisfies the following properties:
\begin{enumerate}
\item $\iota^2=e$ and $\iota\in SO(4)$. \label{firstitem}
\item $\iota(C) = C$.\label{cpreserved}
\item $\iota(\Gamma)=\Gamma$ and $\iota$ fixes no point on $\Gamma$.\label{seconditem}
\item $\iota(W_0)=W_0$. \label{thirditem}
\item $\iota(\overline{\tau}_{1,2(k+1)}\cap W_0)=\overline{\tau}_{1,2(k+1)}\cap W_0$.\label{fourthitem}
\end{enumerate}
To see item (\ref{firstitem}), observe that for any map $m:\mathbb{S}^3\to\mathbb{S}^3$ given by \begin{equation}m(z,w)=(e^{i\alpha}\overline{z},e^{i\beta}\overline{w})\end{equation} with $\alpha, \beta\in\mathbb{R}$ (such as $\iota$) there holds
\begin{equation}
m^2(z,w) = (e^{i\alpha}(\overline{e^{i\alpha}\overline{z})},e^{i\alpha}(\overline{e^{i\alpha}\overline{w})}) = (z,w).
\end{equation}
Recall from Section \ref{symmetries} that the conjugation map $c$ given by $c(z,w)=(\overline{z},\overline{w})$ is in $SO(4)$ and $\iota$ is equal to $c$ composed with a rotation.  Thus $\iota\in SO(4)$.
Item (\ref{cpreserved}) is immediate from \eqref{defiota}.  To see item (\ref{seconditem}), observe that $\iota(B_0)=B_1$, $\iota(B_1)=B_0$ and $\iota(A_j)=A_{m+1-j}$ for each $j$.  To see that $\iota$ fixes no points of $\Gamma$, observe that the fixed point set of any map $m$ as above is \begin{equation}F_m:=\{(e^{i\alpha/2}r_1,e^{i\beta/2}r_2)\in\mathbb{S}^3\;|\; r_1^2+r_2^2=1\}.\end{equation}  Thus the fixed point set of $\iota$ is 
\begin{equation}
F_{\iota}=\{(e^{i \frac{\pi(m+1)}{2m(k+1)}} r_1, \ e^{i \frac{\pi}{m}} r_2)\in\mathbb{S}^3 \;|\; r_1^2+r^2_2=1\}.
\end{equation}
Thus $F_\iota\cap W_0\subset H_{\pi/m}$ which implies that $F_\iota$ is disjoint from the segments of $\Gamma$ between $A$ points and $B$ points (which are in $H_0$ and $H_{2\pi/m}$.  On the other hand we can see directly that $F_\iota\cap C$ is disjoint from the arcs $A_0A_1$ and $A_mA_{1+m}$.
To see item (\ref{thirditem}) observe that if $(z,w)\in H^0$, and thus $\mbox{Im}(w)=0$ then $\mbox{Im}(e^{-i\frac{2\pi}{m}}e^{i\frac{2\pi}{m}}\overline{w})=\mbox{Im}(w)=0$.  Similarly if $\mbox{Re}(w)\geq 0$ then \begin{equation}\mbox{Re}(e^{-i2\pi/m}e^{i2\pi/m}\overline{w})=\mbox{Re}(w)\geq 0.\end{equation}  Recalling the definition of $H^i$, this implies $\iota(H^0)=H^1$.  Similarly, we see $\iota(H^1)=H^0$. Item (\ref{fourthitem}) follows trivially.  This completes the verification of the desired properties of $\iota$.

Consider the group 
\begin{equation}\label{GkmGrp}
    G_{k,m}\subset SO(4)
\end{equation} generated by the rotation $\rho_{\frac{2\pi}{k+1}}$ and involution $\iota$.  The group $G_{k,m}$ is isomorphic to the dihedral group $\mathbb{D}_{k+1}$ with $2(k+1)$ elements since we can compute directly 
\begin{equation}\label{relation}
\iota\circ \rho_{\frac{2\pi}{k+1}}\circ\iota =\rho^{-1}_{\frac{2\pi}{k+1}},
\end{equation}
while $\iota^2=e$ and $\rho^{k+1}_{\frac{2\pi}{k+1}}=e$.

Let us decompose the set
\begin{equation}
W_0\setminus \overline{\tau}_{1,2(k+1)}=W_0^1\cup W_0^2\cup...\cup W_0^{2(k+1)}
\end{equation}
into $2(k+1)$ pairwise disjoint ``chambers," where
\begin{equation}
W_0^j:=\rho_{\frac{\pi j}{k+1}}(W_0^1), 
\end{equation}
and the labeling is chosen so that
\begin{equation}\label{whereisgamma}
\Gamma\subset \partial \overline{W}_0^1.
\end{equation}

Moreover, for each $j=1,...,2(k+1)$, there exists $\iota\neq e\in G_{k,m}$ with $\iota^2=e$ so that 
\begin{equation}
\iota(W_0^j)=W^j_0.
\end{equation}

The closure of each chamber $\overline{W}_0^i$ is mean convex.  Indeed, its boundary is  piecewise minimal, two sides of which are pieces of minimal hemispheres  $P_0\subset H^0$ and $P_1\subset H^1$ and two sides $S_1$ and $S_2$ of which are pieces of $\overline{\tau}_{1,2(k+1)}$ with boundary
\begin{equation}
\partial S_1 =B_0A_0A_1B_1B_0\mbox{ and } \partial S_2 = B_1A_{1+m}A_mB_0B_1.
\end{equation}
and so that 
\begin{equation}
S_1\cap S_2=\partial S_1\cap \partial S_2=B_0B_1.
\end{equation}
Let us verify that the angles where these faces meet are all at most $\pi$.  Recalling the computation in Section \ref{immersecsection}, along the geodesic segment $B_0B_1$ on $C^*$, the surfaces $S_1$ and $S_2$ meet at the constant angle $\frac{\pi}{k+1}$.  Since $S_1$ and $S_2$ are constrained to the convex wedge $W_0$, it is clear \emph{a priori} that the angles at which they intersect $P_0$ and $P_1$ are constrained to meet at angles at most $\pi$.  For the sake of completeness, let us describe the dihedral angles.  Along $A_0A_1$, $P_0$ and $S_1$ meet at angles increasing monotonically from $0$ to $\frac{2\pi}{m}$. Along $A_1B_1$, $P_1$ and $S_1$ meet at angles beginning at $\pi$ and descending to $\pi/2$ monotonically.   Along $B_0A_0$, the surfaces $S_1$ and $P_0$ meet at angles descending monotonically from $\pi/2$ to $0$.  The intersection of $S_2$ with $P_0$ and $P_1$ are analogous.  Since $m\geq 3$, all the angles enumerated above are at most $\pi$ implying $\overline{W}_0^1$ is (piecewise) mean convex. 

Since by \eqref{whereisgamma}, the contour $\Gamma$ lies on the boundary of a mean convex domain, we can apply Theorem 6.1 in \cite{J} to find an embedded minimal disk $\mathcal{P}$ satisfying
\begin{enumerate}[(a)]
\item $\mathcal{P}\subset\overline{W}_0^1$, 
\item $\partial\mathcal{P}=\Gamma$, 
\item $\iota(\mathcal{P})=\mathcal{P}$.\label{whyinvariant}
\end{enumerate}
Item (\ref{whyinvariant}) follows from the proof of a result of Meeks-Yau (\cite{MYD}, Theorem 4) establishing that an area-minimizing disk in a three-manifold is invariant under any finite group of orientation-preserving isometries of the ambient manifold preserving and acting freely on the boundary of the disk.  The involution $\iota$ indeed preserves the contour $\Gamma$ and acts freely on it by item (\ref{seconditem}).

  \begin{figure}
	\centering
	\resizebox{4in}{!}{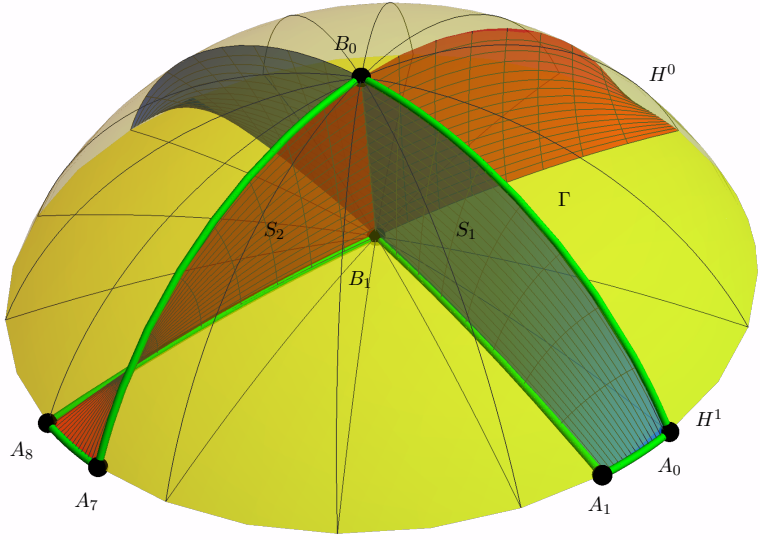}
	\caption{Stereographic projection of lens along with $W_0\cap\bar{\tau}_{1,4}$ and $\Gamma$ for $m=7$ and $k=1$.}
	\label{Plotm7}
\end{figure}

Recalling \eqref{howtau}, we can reparameterize $\overline{\tau}_{1,2(k+1)}\cap W_0$ as a $\mathbb{D}_{k+1}$-equivariant union of $(k+1)$ piecewise smooth disks with boundary in $W_0$ whose singular sets are a crease along the geodesic segment $C^*\cap W_0$:
\begin{equation}\label{oy}
\overline{\tau}_{1,2(k+1)}\cap W_0= \bigcup_{j=1}^{k+1} \rho_{\frac{2\pi j}{k+1}}(\overline{\tau}_{1,2(k+1)}\cap \overline{W}_0^1).
\end{equation}
From \eqref{oy} it follows that
\begin{equation}\label{areabound}
\mbox{Area}(\mathcal{P})< \frac{\mbox{Area}(\overline{\tau}_{1,2(k+1)}\cap W_0)}{k+1}, 
\end{equation}
as if there were equality, each $\rho_{\frac{2\pi i}{k+1}}(\overline{\tau}_{1,2(k+1)}\cap \overline{W}_0^1)$ would be a non-smooth area-minimizing disk, an impossibility.  Again because the area-minimizing disk $\mathcal{P}$ is smooth, it follows that 
\begin{equation}
\mathcal{P}\cap \partial \overline{W}_0^1=\partial\mathcal{P}.
\end{equation}

\subsection{Filling the initial lens}
We show that $m(k+1)$ successive reflections of the fundamental patch $\mathcal{P}$ generate an embedded surface $\tilde{\xi}_{k,m}$ with boundary $C$ that occupies one half of a partition of $\mathbb{S}^3$ into $2m(k+1)$ congruent regions.

First set
\begin{equation}\label{q1}
\mathcal{Q}_1:=\bigcup_{p=0}^{k}\rho_{\frac{2\pi p}{k+1}}(\mathcal{P}). 
\end{equation}
The isometric copies of $\mathcal{P}$ comprising $\mathcal{Q}_1$ fill in half of the chambers of the lens $W_0$. We claim that $\mathcal{Q}_1$ is $G_{k,m}$-equivariant. To see this, note that $\mathcal{Q}_1$ is manifestly invariant under $\rho_{\frac{2\pi}{k+1}}$.  To see that it is invariant under $\iota$, observe that by the relation \eqref{relation} in $G_{k,m}$ we get 

\begin{equation}\label{relation2}
\iota\circ\rho_{\frac{2\pi p}{k+1}}=\iota\circ \rho^p_{\frac{2\pi}{k+1}}=\rho^{-p}_{\frac{2\pi }{k+1}}\circ\iota=\rho_{-\frac{2\pi p}{k+1}}\circ\iota \end{equation} 
Thus since $\mathcal{P}$ is $\iota$-invariant by construction, applying \eqref{relation2} we get
\begin{equation}
\iota(\mathcal{Q}_1)=\bigcup_{p=0}^{k}\iota\circ \rho_{\frac{2\pi p}{k+1}}(\mathcal{P})=\bigcup_{p=0}^{k}\rho_{-\frac{2\pi p}{k+1}}(\mathcal{P})=\mathcal{Q}_1.
\end{equation}
Finally note that
\begin{equation}\label{firstit}
\mathcal{Q}_1\cap H_0=\bigcup_{p=0}^k \rho_{\frac{2\pi p }{k+1}}(A_0B_0\cup A_{1+m}A_{m}\cup A_0A_1\cup B_0A_{m}), 
\end{equation}
and
\begin{equation}\label{secondit}
\mathcal{Q}_1\cap H_1=\bigcup_{p=0}^k \rho_{\frac{2\pi p }{k+1}}(A_1B_1\cup A_{1+m}A_{m}\cup A_0A_1\cup B_1A_{1+m}).  
\end{equation}
Using \eqref{pointrotate}, we see that in $\mathcal{Q}_1\cap\mbox{int}(H_0)$ (resp. $\mathcal{Q}_1\cap\mbox{int}(H_1)$) the enumerated segments in \eqref{firstit} and \eqref{secondit} pair off to comprise $k+1$ equally-spaced geodesic segments of length $\pi$ meeting at $B_0$ (resp. $B_1$).  Indeed, 
\begin{equation}\label{pairingoff}
\mathcal{Q}_1\cap\mbox{int}(H_0)=\bigcup_{p=0}^k (A_{2pm}B_0)\cup(B_0A_{m(1+2p)}), 
\end{equation}
and thus each point in the list $\{A_0,A_m,A_{2m},...,A_{(2k+1)m}\}$ pairs off with its antipode in the same list as the antipodal point to $A_{qm}$ is $A_{qm+m(k+1)}=A_{m(q+k+1)}$. This completes the discussion of filling out half of the chambers in the lens $\mathcal{Q}_1$. 

\subsection{Reflections}
Given two non-antipodal points $A$ and $B$ in $\mathbb{S}^3$, denote by $R_{AB}$ the $\pi$-rotation through the unique geodesic containing $A$ and $B$.  Note that a direct computation gives
\begin{equation}\label{rind}
R_{A_iB_j}(A_kB_l) = A_{2i-k}B_{2j-l}, 
\end{equation}
where the integer label of $A$ is understood modulo $2m(k+1)$ and the integer label of $B$ modulo $m$.
We now describe an inductive procedure to fill out the remaining lenses $W_1,...,W_{m-1}$. Given $\mathcal{Q}_i$, we define $\mathcal{Q}_{i+1}$ inductively by
\begin{equation}\label{defind}
\mathcal{Q}_{i+1}:=R_i(\mathcal{Q}_i), 
\end{equation}
where for $i\geq 0$, $R_i$ is the $\pi$-rotation given by 
\begin{equation}\label{schwarz}
R_{i+1}=R_{B_{i+1}A_{(i+1)(1+m)}}.
\end{equation}
The index of $A$ in \eqref{schwarz} is interpreted modulo $2m(k+1)$.
We claim
\begin{equation}
\tilde{\xi}_{k,m}:=\bigcup_{i=1}^m \mathcal{Q}_i.
\end{equation}
is a smooth embedded minimal surface in $\mathbb{S}^3$ with boundary $C$.

We first claim by induction that \begin{equation}\label{axiswhere}B_{i}A_{i(1+m)}\subset \partial\mathcal{Q}_i \cap H_i.\end{equation}

The case $i=1$ is established by \eqref{firstit}.  Assume by induction that $B_{i}A_{i(1+m)}\subset \partial\mathcal{Q}_{i} \cap H_i$ and so $B_{i-1}A_{(i-1)(1+m)}\subset \partial\mathcal{Q}_{i}\cap H_{i-1}$.  On the one hand,  by the inductive definition \eqref{defind}:
\begin{equation}
R_{i}(B_{i-1}A_{(i-1)(1+m)})\subset\partial \mathcal{Q}_{i+1}\cap H_{i+1}.\end{equation}  
On the other hand, by \eqref{rind}, we have
\begin{equation}
R_{i}(B_{i-1}A_{(i-1)(1+m)})=R_{B_iA_{i(1+m)}} (B_{(i-1)A_{(i-1)(1+m))}})=   B_{i+1}A_{(i+1)(1+m)}, 
\end{equation}
completing the induction.

From \eqref{axiswhere} we see that $\mathcal{Q}_{i+1}$ is obtained by Schwarz reflection along one of the $k+1$ boundary segments of $\mathcal{Q}_i$ in $\mbox{int}(H_i)$.   We will show that by the involutive symmetry of $\mathcal{P}$, the Schwarz reflection \eqref{schwarz} is in fact a translation.  Toward that end, let us denote the translation
\begin{equation}
t=\sigma_{\frac{2\pi}{m}}\circ \rho_{\frac{2\pi (m+1)}{2m(k+1)}}.
\end{equation}
Then we can also express:
\begin{equation}\label{newrot}
R_i=t\circ R_{i-1}\circ t^{-1}.
\end{equation}
since the axes of the two rotations differ by the translation $t$.
Moreover, by construction, the surface $\mathcal{Q}_i$ is invariant under the dihedral group given inductively by \begin{equation}G^{i}_{k,m}=R_{i-1}\circ G^
{i-1}_{k,m}\circ R_{i-1}^{-1}.\end{equation} conjugate to $G^1_{k,m}=G_{k,m}$.

Note that by direct computation, 
\begin{equation}
R_1=t\circ \iota
\end{equation}
Define involutions $\iota_i\in G_{k,m}^i$ (setting $\iota_1=\iota$) inductively by
\begin{equation}\label{iota}
\iota_i=R_{i-1}\circ \iota_{i-1}\circ R^{-1}_{i-1}
\end{equation}
We claim for all $i$, 
\begin{equation}\label{want}
R_{i}=t\circ \iota_i.
\end{equation}
The case $i=1$ was established above.  Assume by induction
\begin{equation}\label{generalinvariance}
R_{i-1}=t\circ \iota_{i-1}.
\end{equation} 
Plugging this into \eqref{iota} we get (since $\iota_{i-1}$ is an involution)
\begin{equation}\label{iotasimple}
\iota_i=t\circ \iota_{i-1}\circ\iota_{i-1}\circ\iota_{i-1}\circ t^{-1}=t\circ\iota_{i-1}\circ t^{-1}. 
\end{equation}
Plugging \eqref{newrot} and \eqref{iotasimple} into the inductive assumption \eqref{generalinvariance} we obtain
\begin{equation}
t^{-1}\circ R_i\circ t =t\circ (t^{-1})\circ \iota_{i} \circ t
\end{equation}
which after rearranging gives \eqref{want}
as desired. 

Since $\mathcal{Q}_i$ is invariant under $G_{k,m}^i$ we get from \eqref{generalinvariance}
\begin{equation}\label{iteration}
\mathcal{Q}_{i+1} = t\circ\iota_i (\mathcal{Q}_i)= t(\mathcal{Q}_i)= \sigma_{\frac{2\pi}{m}}\circ \rho_{\frac{2\pi (m+1)}{2m(k+1)}}(\mathcal{Q}_i).
\end{equation}
\noindent
Thus $\mathcal{Q}_{i+1}$ arises from $\mathcal{Q}_i$ from both Schwarz reflecting as well as a translation.  

Note from \eqref{pairingoff} that $\mathcal{Q}_i\cap\mbox{int}(H_i)=\mathcal{Q}_{i+1}\cap\mbox{int}(H_i)$ for each $i$ since rotating by $\pi$ about one of $k+1$ equally spaced geodesic segments in $\mathcal{Q}_i\cap\mbox{int}(H_i)$ preserves the collection of segments.  The choice of segment over which we Schwarz reflect to obtain $\mathcal{Q}_{i+1}$ does not alter the resulting surface.  Indeed, if $\tilde{R}_i$ denotes the $\pi$ rotation about one of the \emph{other} geodesic segments $B_iA_{pm+i(1+m)}$
in $\mathcal{Q}_i\cap\mbox{int}(H_i)$ for some $p$, a short computation gives
\begin{equation}
\tilde{R}_i\circ R_i = \rho_{\frac{2\pi p}{k+1}}.
\end{equation}
Thus since $\mathcal{Q}_i$ is invariant under $ \rho_{\frac{2\pi p}{k+1}}$ we obtain
\begin{equation}
\tilde{R}_i\circ R_i(\mathcal{Q}_i)=\rho_{\frac{2\pi p}{k+1}}(\mathcal{}Q_i)=\mathcal{Q}_i, 
\end{equation}
so that
\begin{equation}\label{ind}
R_i(\mathcal{Q}_i)=\tilde{R}_i(\mathcal{Q}_i)=\mathcal{Q}_{i+1}.
\end{equation}

Since $\mathcal{Q}_{i+1}$ arises from $\mathcal{Q}_i$ by Schwarz reflection over a boundary segment and by \eqref{ind} the choice of boundary segment over which one Schwarz reflects gives the same surface, it follows that $\mathcal{Q}_{i+1}$ extends $\mathcal{Q}_i$ smoothly over each of the $2(k+1)$ equally-spaced geodesic segments $\mbox{int}(H_i)\setminus \{B_i\}$.  By the removability of singularities theorem, it is also smooth over the point $B_i$. 

Thus $\tilde{\xi}_{k,m}$ is a smooth, connected, embedded minimal surface away from its intersection in $\mbox{int}(H_0)$ where it might fail to close up.  Let us show that it indeed closes up.  This step uses the involutive symmetry and in particular its consequence \eqref{iteration}.  Indeed, after $m$ iterations, we get from \eqref{iteration}
\begin{equation}\label{last}
\mathcal{Q}_{m+1}=\rho_{\frac{2\pi(m+1)}{2(k+1)}}(\mathcal{Q}_1).
\end{equation}
Since $m$ is odd, $\frac{m+1}{2}$ is an integer and thus from from \eqref{q1} and \eqref{last} we deduce that
\begin{equation}
\mathcal{Q}_{m+1}=\mathcal{Q}_1.
\end{equation}
In other words, the surface $\tilde{\xi}_{k,m}$ extends smoothly over $H_0$ to close up.

By \eqref{areabound} we obtain
\begin{equation}\label{areaboundagain}
\mbox{Area}(\tilde{\xi}_{k,m})=m(k+1)\mbox{Area}(\mathcal{P})< \mbox{Area}(\overline{\tau}_{1,2(k+1)}), 
\end{equation}
establishing (\ref{item3}).  Item (\ref{item5}) follows directly from the choice of boundary curve of the fundamental patch $\mathcal{P}$.  

\subsection{Genus and Euler number}
Let us verify item (\ref{genus}).  Note that the six-sided geodesic polygon $\Gamma$ has $4$ interior angles equal to $\frac{\pi}{2}$ (at the vertices on $C$) and the other two interior angles are equal to $\frac{\pi}{k+1}$ (at the vertices on $C^*$).  Thus four of the exterior angles are $\frac{\pi}{2}$ and two exterior angles are $\frac{k\pi}{k+1}$. 

By the Gauss-Bonnet formula applied to the minimal surface $\mathcal{P}$, we obtain
\begin{equation}
2\pi+\frac{2k\pi}{k+1}+\int_\mathcal{P} K = 2\pi\chi(\mathcal{P}).  
\end{equation}
Since $\mathcal{P}$ is a disk, we get $\chi(\mathcal{P})=1$, and thus
\begin{equation}\label{each}
\int_\mathcal{P} K=-\frac{2k\pi}{k+1}.
\end{equation}
As the $m(k+1)$ pieces $\{\mathcal{P}_i\}_{i=1}^{m(k+1)}$ comprising the surface $\tilde{\xi}_{k,m}$ are pairwise isometric, we obtain from \eqref{each}
\begin{equation}\label{local}
\int_{\tilde{\xi}_{k,m}}K =\sum_{i=1}^{m(k+1)}\int_{\mathcal{P}_i} K=m(k+1)\int_{\mathcal{P}}K=-2\pi km.
\end{equation}
On the other hand, the Euler characteristic of $\xi_{k,m}$, a non-orientable surface with one boundary component is $1-\mbox{genus}(\tilde{\xi}_{k,m})$.  Thus by the Gauss-Bonnet formula again and \eqref{local}, 
\begin{equation}
-km = 1-\mbox{genus}(\tilde{\xi}_{k,m}), 
\end{equation}
whence
\begin{equation}
\mbox{genus}(\tilde{\xi}_{k,m})=1+km, 
\end{equation}
as desired.

Note that $e(\tilde{\xi}_{k,m})=2(k+1)$ because by construction, in each patch the total variation in the direction of the conormal along each of the patch's two segments on $C$ is $\frac{2\pi}{m}$ and there are $m(k+1)$ patches.  

This completes the verification of item (\ref{genus}).

\subsection{Symmetry and limit}
Let us verify item (\ref{item4}). We obtain from the inductive definition \eqref{iteration} that the minimal surface $\tilde{\xi}_{k,m}$ is invariant under $g_1,g_2\subset SO(4)$ given by
\begin{equation}
g_1:=t=\sigma_{\frac{2\pi}{m}}\circ\rho_{\frac{2\pi(m+1)}{2m(k+1)}}\mbox{ and } g_2=\rho_{\frac{2\pi}{k+1}}.
\end{equation}  
\noindent
To compute the symmetry group $G$ generated by $g_1$ and $g_2$, consider the map
\begin{equation}
\phi:\mathbb{Z}^2\to SO(4) 
\end{equation}
given by $\phi(r,s)=g_1^rg_s^s$.  Then 
\begin{equation}
G\cong\mbox{Im}(\phi)\cong \mathbb{Z}^2/\mbox{Ker}(\phi), 
\end{equation}
and $|G|$, the order of the group $G$, is the area of the parallelogram $\mathcal{P}(v_1,v_2)$ formed by the two spanning basis vectors $v_1$ and $v_2$ of the lattice $\mbox{Ker}(\phi)$.
To determine the lattice $\mbox{Ker}(\phi)$, suppose $s,t\in\mathbb{Z}$ satisfy $g_1^rg_2^s=e$.  We readily obtain (letting $a=(m+1)/2$):
\begin{equation}
\frac{r}{m}\in\mathbb{Z}\mbox{ and } \frac{ra}{m(k+1)}+\frac{s}{k+1}\in\mathbb{Z}.
\end{equation}
Thus, for integers $A$ and $B$ we get
\begin{equation}
r=mA\mbox{ and } Aa+s=B(k+1)
\end{equation}
so
\begin{equation}
(r,s) = (mA, B(k+1)-Aa)=A(m,-a)+B(0,k+1).
\end{equation}
The lattice $\mbox{Ker}(\phi)$ is thus generated by $v_1=(m,-a)$ and $v_2=(0,k+1)$ so that 
\begin{equation}
|G|=\mbox{Area}(\mathcal{P}(v_1,v_2))= \begin{vmatrix}
m & -a \\
0 & k+1
\end{vmatrix}=m(k+1).
\end{equation}
Note that 
\begin{equation}\label{goodelement}
g_3:=g_1^2g_2^{-1}=\sigma_{\frac{4\pi}{m}}\circ\rho_{\frac{2\pi}{m(k+1)}}=\sigma_{\frac{2\pi(2(k+1))}{m(k+1)}}\circ\rho_{\frac{2\pi}{m(k+1)}}\in G
\end{equation}
and the order of $g_3$ is clearly $m(k+1)$. Since $|G|=m(k+1)$ this implies that $G$ is the cyclic group $\mathbb{Z}_{m(k+1)}$.  

Note that the involution $\iota$ satisfies \begin{equation}\iota \circ g_3\circ  \iota=g^{-1}_3.\end{equation}  Thus since $g_3^{m(k+1)}=e$, the subgroup of $SO(4)$ preserving $\tilde{\xi}_{k,m}$ generated by $\iota$ and $g_3$ is the dihedral group $\mathbb{D}_{m(k+1)}$, establishing item (\ref{item4}).

Let us now verify item (\ref{item6}).  The equation \eqref{goodelement} implies that $\tilde{\xi}_{k,m}$ is invariant under a $\mathbb{Z}_{m(k+1)}$ subaction of the $\mathbb{S}^1$-action on $\mathbb{S}^3$:
\begin{equation}\label{action}
e^{i\theta}(z,w) = ( e^{i\theta}z, e^{2i(k+1)\theta}w).
\end{equation}
Recalling Section \ref{symmetries}, the surface $\overline{\tau}_{1,2(k+1)}$ is invariant under the full $\mathbb{S}^1$-action given in \eqref{action}.

Let $\tau:\mathbb{S}^3\to\mathbb{S}^3$ denote the Schwarz reflection about $C$. Let \begin{equation}D(\tilde{\xi}_{k,m})=\tilde{\xi}_{k,m}\cup\tau(\tilde{\xi}_{k,m})\end{equation} denote the smooth closed minimal surface in $\mathbb{S}^3$ containing $C$.  Then $D(\tilde{\xi}_{k,m})$ is a stationary integral varifold in $\mathbb{S}^3$.  If $m_i\to\infty$ is a sequence of odd numbers, then since the areas of $\tilde{\xi}_{k,m_i}$ are bounded uniformly by \eqref{areaboundagain}, Allard's Compactness Theorem implies after passing to a subsequence (not relabelled) the sequence $D(\tilde{\xi}_{k,m_i})$ converges as varifolds to a stationary integral varifold $V_k$. Moreover by \eqref{areaboundagain}
\begin{equation}\label{limitbound}
||V_k||=\lim_{i\to\infty}2||\tilde{\xi}_{k,m_i}||\leq 2\mbox{Area}(\overline{\tau}_{1,2(k+1)}).
\end{equation}

In addition, the support and multiplicity function of the limiting varifold $V_k$ are invariant under the $\mathbb{S}^1$-action described in \eqref{action}.

On the other hand, by item (\ref{item5}), the support of $V_k$ contains 
$\mbox{supp}(\tau_{1,2(k+1)})$ and this set occurs with multiplicity at least $1$ by integrality of $V_k$.  In light of \eqref{limitbound}, we obtain item  (\ref{item6}).  This completes the proof of Theorem \ref{main}.

\section{Incompressible surfaces in lens spaces}\label{lenssection}

Recall that if $p$ and $q$ are relatively prime, the quotient of the free $\mathbb{Z}_p$-action on $\mathbb{S}^3$ generated by 
\begin{equation}
\xi_{p,q}(z,w)=(e^{2\pi i/p}z,e^{2\pi i q/p}w)
\end{equation}
is the lens space $L(p,q)$.  Since the $\mathbb{Z}_p$-action defined above also depends on $q$, we call it the $(p,q)$-action.

Let \begin{equation}\pi_{p,q}:\mathbb{S}^3\to L(p,q)\end{equation} denote the projection map.  

For even $p=2k$, since $H_2(L(2k,q),\mathbb{Z}_2)\cong\mathbb{Z}_2$ but $H_2(L(2k,q),\mathbb{Z})\cong 0$  there exist homologically non-trivial non-orientable surfaces in $L(2k,q)$.  In 1969, Bredon-Wood \cite{Bre} determined the infimal genus of such a surface that is realized by an embedding in terms of data from the iterated fraction expansion of the ratio $p/q$. 

For $p$ odd, on the other hand, we have the following analogous situation on which our construction is based:
\begin{lemma}\label{disc}
If $p$ is odd, then any generator of $\pi_1(L(p,q))$ bounds a non-orientable surface but no orientable surface. 
\end{lemma}
\begin{proof}
Since $\pi_1(L(p,q))\cong\mathbb{Z}_p$ is abelian, we get that $H_1(L(p,q),\mathbb{Z})\cong\mathbb{Z}_p$.  By the Universal Coefficient Theorem (Theorem 6.2 in \cite{Massey}), we get (since the Tor term vanishes) 
\begin{equation}
H_1(L(p,q);\mathbb{Z}_2)\cong H_1(L(p,q),\mathbb{Z}))\otimes\mathbb{Z}_2 \cong\mathbb{Z}_p\otimes\mathbb{Z}_2\cong\mathbb{Z}_{\mbox{gcd}(p,2)}\cong 0, 
\end{equation}
since $p$ is odd.

Since $H_1(L(p,q);\mathbb{Z}_2)\cong0$ but $H_1(L(p,q);\mathbb{Z})\cong\mathbb{Z}_p$, it follows that a generator of $H_1(L(p,q),\mathbb{Z})\cong\pi_1(L(p,q))$ bounds a non-orientable surface but no orientable surface.  
\end{proof}

We thus obtain: 
\begin{thm}
For $p,q\in\mathbb{N}$ ($p\geq 2$ and $1\leq q< p$) such that $\mbox{gcd}(p,q)=1$ and $p$ is odd, there exists a non-orientable embedded minimal surface $\mu_{p,q}$ in $\mathbb{S}^3$ with boundary a great circle.  Moreover, the surface $\mu_{p,q}$ is invariant under the $(p,q)$-action.
\end{thm}
\begin{proof}
Consider the circle \begin{equation}C_{p,q}:=\pi_{p,q}(C),\end{equation} which is a closed geodesic in $L(p,q)$ of length $2\pi/p$ which generates the fundamental group of $L(p,q)$.  By a result of R. Hardt (\cite{Hardt}), there exists a smoothly embedded area-minimizing mod $2$ flat chain $\Sigma_{p,q}\subset L(p,q)$ with boundary equal to $C_{p,q}$.  Since $C_{p,q}$ represents a generator of $\pi_1(L(p,q))$, it follows from Lemma \ref{disc} that $\Sigma_{p,q}$ is a non-orientable minimal surface. Its lift
\begin{equation}
\mu_{p,q}:=\pi^{-1}_{p,q}(\Sigma_{p,q}), 
\end{equation}
is an embedded minimal surface in $\mathbb{S}^3$ with boundary the great circle $C\subset\mathbb{S}^3$.  By a result of Hardt-Simon (Theorem \ref{orientable}), $\mu_{p,q}$ is either non-orientable or a hemisphere.  
Since disks cannot cover non-orientable surfaces, it follows that $\mu_{p,q}$ is also non-orientable.  

\end{proof}

We show that suitable minimal surfaces $\tilde{\xi}_{k,m}$ indeed descend to the lens spaces $L(p,q)$:

\begin{prop}\label{crosscap}
Let $p$ and $q$ be relatively prime integers with $1\leq q<p$ and $p$ odd with $p\geq 3$. The following hold: 
\begin{enumerate}
    \item If $q$ is even, then the surface $\pi_{p,q}(\tilde{\xi}_{p,q/2-1})$ is a minimal embedding in $L(p,q)$ with boundary $C_{p,q}$.
    \item If $q$ is odd, then the surface $\pi_{p,q}(c_2(\tilde{\xi}_{p,(p-q)/2-1}))$ is a minimal embedding in $L(p,q)$ with boundary $C_{p,q}$.
\end{enumerate}
\end{prop}
\begin{proof}
Recall from \eqref{goodelement} that $\tilde{\xi}_{k,m}$ is invariant under the group of isometries generated by
\begin{equation}
g_3=\sigma_{\frac{4\pi}{m}}\circ\rho_{\frac{2\pi}{m(k+1)}}.
\end{equation}
Thus the surface $\tilde{\xi}_{k,m}$ is also invariant under the element
\begin{equation}
g_3^{k+1}=\sigma_{\frac{2\pi(2k+2)}{m}}\circ\rho_{\frac{2\pi}{m}}, 
\end{equation}
which generates the $(m,2(k+1))$-action on $\mathbb{S}^3$ when $2(k+1)$ is relatively prime to $m$. Thus if $q$ is even, $\tilde{\xi}_{p,q/2-1}$ indeed descends to $L(p,q)$ as an embedded minimal surface.  Moreover, we have from item 
(\ref{genus}) in Theorem \ref{main} that the genus of $\tilde{\xi}_{p,q/2-1}$ is $pq/2-p+1$.  Since
\begin{equation}
\chi(\tilde{\xi}_{p,q/2-1})=p\chi(\pi_{p,q}(\tilde{\xi}_{p,q/2-1})), 
\end{equation}
and the Euler characteristic of a non-orientable surface with one boundary component and $r$ cross-caps added to a disk is $1-r$ we get that
\begin{equation}
p-pq/2 = p\chi(\pi_{p,q}(\tilde{\xi}_{p,q/2-1})).
\end{equation}
Thus we obtain
\begin{equation}
\mbox{genus}(\pi_{p,q}(\tilde{\xi}_{p,q/2-1}))=q/2.
\end{equation}

Let us now address item (2), the case when $q$ is odd.  Observe that by the preceding discussion, the surface $c_2(\tilde{\xi}_{k,m})$ (recalling Section \ref{symmetries}) is invariant under the $(m,m-2(k+1))$-action. Thus, if $q$ is odd, $c_2(\tilde{\xi}_{p,(p-q)/2-1})$ descends to the lens space $L(p,q)$.  The computation of the genus is analogous to the even case.
\end{proof}
\begin{rmk}
Recalling \eqref{secondconj}, observe that the surface $c_2(\tilde{\xi}_{k,m})$ has opposite chirality to $\tilde{\xi}_{k,m}$ in the sense that their boundary slopes differ by a sign.
\end{rmk}

\begin{rmk}\label{coincide}
It is natural to expect that $\mu_{p,q}=\pi_{p,q}(\tilde{\xi}_{p,q/2-1})$ for $q$ even and $\mu_{p,q}=\pi_{p,q}(c_2(\tilde{\xi}_{p,(p-q)/2-1}))$ for $q$ odd.  One can minimize area in the homotopy class of the surfaces posited in Proposition \ref{crosscap} to obtain a (possibly immersed) minimal surface with boundary $C_{p,q}$ realizing this genus. Since such minimizers are homotopic to embeddings, in light of the work of Freedman-Hass-Scott \cite{FHS}, one expects that the homotopy minimizers indeed coincide with the geometric measure theory minimizer $\mu_{p,q}$ (and also suitable $\tilde{\xi}_{k,m}$ surfaces).  Note that the minimization procedure is necessarily not conducted in a mean convex domain of $\mathbb{S}^3$.
\end{rmk}

\section{Admissible genera}\label{admissiblesection}
In this section we collect results of Whitney and Massey that give the possible combinations of genera and Euler numbers can be realized by smooth embeddings with boundary a great circle. 

If $\Sigma$ is a non-orientable embedded surface of genus $g$ in $\mathbb{S}^3$ with boundary an unknot $K$, then capping it off with a disk in $B^4$ we obtain a closed non-orientable surface $\tilde{\Sigma}$ embedded in $\mathbb{R}^4$ with  $\chi(\tilde{\Sigma})=2-g$.  

Whitney \cite{Whitney} proved that the normal Euler number $e(\tilde{\Sigma})=e(\Sigma)$ (see Section 2 in \cite{Conway} for this equality or \cite[1.1.1]{Yasuhara}) and Euler characteristic $\chi(\tilde{\Sigma})$ obey the following relation 
\begin{equation}
e(\tilde{\Sigma})=2\chi(\tilde{\Sigma}) \;\;\;\mbox{   mod } 4, 
\end{equation}
which gives
\begin{equation}
e(\Sigma)=-2g \;\;\;\mbox{   mod } 4.
\end{equation}

We deduce the following:

\begin{lemma}\label{parity}
If $\Sigma\subset \mathbb{S}^3$ is a non-orientable embedded surface with boundary an unknot with $e(\Sigma)=2k$, then the genus of $\Sigma$ has the same parity as the integer $k$.
\end{lemma}

Massey \cite{Ma} -- see also \cite[Corollary 1.1.1]{Yasuhara} -- later proved that 
\begin{equation}\label{options}
e(\Sigma)\in \{-2g,-2g+4,..., 2g-4, 2g\}.
\end{equation}

We obtain: 
\begin{thm}[Admissible genera \cite{Whitney,Ma}]\label{admissible}
There exists a smooth embedded non-orientable surface $\Sigma$ with boundary an unknot in $\mathbb{S}^3$ if and only if
\begin{equation}\label{adgen}
\mbox{genus}(\Sigma)=\frac{e(\Sigma)}{2}+2n\mbox{ where } n\in\mathbb{N}\cup\{0\}.
\end{equation}
In particular for such a surface $\Sigma$,
\begin{equation}
\mbox{genus}(\Sigma)\geq\frac{e(\Sigma)}{2}.
\end{equation}
\end{thm}

\begin{proof}
Lemma \ref{parity} and \eqref{options} establish the necessity of \eqref{adgen} for such a surface $\Sigma$ to exist.  Given $k\geq 1$, we will show on the other hand that there exists a smooth embedded surface of genus $k$ and Euler number $2k$ and boundary an unknot.  Since one can then add $n$ handles to this surface (which increases the genus by $2n$ but does not change the Euler number), this shows the sufficiency of \eqref{adgen}.

Let us construct the desired surface in our setting.  Fix $k\geq 1$. Consider the immersed band $\overline{\tau}_{1,2k}$ restricted to a neighborhood $T_{\pi/4}(C^*)$ of the curve $C^*$ where it fails to be an embedding.  The solid torus $T_{\pi/4}(C^*)$ is foliated by the meridian disks $\{H_\theta\cap  T_{\pi/4}(C^*)\}_{t\in [0,2\pi]}$.  The cross sections \begin{equation}F_\theta = H_\theta\cap T_{\pi/4}(C^*)\cap \overline{\tau}_{1,2k}\end{equation} consist of $2k$ equally spaced geodesic segments meeting at the center point $H_\theta\cap C^*$ of $H_\theta$.  

For $\delta$ small enough, let $B_\delta\subset H_0$ be the geodesically convex ball of radius $\delta$ in $H_0$ based at its center.  Then 
\begin{equation}
(\bigcup_{j=1}^{2k} \gamma_{0,\frac{2\pi j}{k}})\cap \partial B_\delta = p_1(\delta)\cup...\cup p_{2k}(\delta), 
\end{equation}
where $p_1(\delta),...,p_{2k}(\delta)$ are equally spaced points on $\partial B_\delta$ varying continuously with $\delta$.
Let $g_i(\delta)$ denote the unique minimizing geodesic segment in $H_0$ between $p_i(\delta)$ and $p_{i+1}(\delta)$ (where $g_{2k}(\delta)$ is the unique minimizing geodesic between $p_{2k}(\delta)$ and $p_1(\delta)$).  Note that the interiors of the segments $\{g_i(\delta)\}_{i=1}^{2k}$ are pairwise disjoint since they are minimizing, the segments vary continuously in $\delta$, and are contained in $B_\delta$.

Let us denote the piecewise smooth geodesic contour:
\begin{equation}
G(\delta) = \bigcup_{i=1}^{2k} g_i(\delta).
\end{equation}

For $\epsilon<\frac{\pi}{4}$, let  $\{\Psi_t\}_{t\in [0,\epsilon]}\subset F_0$ be the family of curves given by 
\begin{equation}
\Psi_t = G(t)\cup \bigcup_{j=1}^{2k} \gamma_{0,\frac{2\pi j}{k}}\cap (F_0\setminus B_t) 
\end{equation}

Let us extend $\{\Psi_t\}_{t\in [0,\epsilon]}$ to a family $\{\Gamma_t\}_{t\in [0,2\pi]}$ also contained in the disk $F_0$ by setting
\begin{equation}
\Gamma_t = 
\begin{cases} 
  \Psi_t & \text{if } t\in [0,\epsilon]\\ 
  
  \Psi_\epsilon& \text{if } t\in [\epsilon,2\pi-\epsilon]\\
  \Psi_{2\pi-t} & \text{if } t\in [2\pi-\epsilon,2\pi]. \\
\end{cases}
\end{equation}
Note that for all $t\in [0,2\pi]$
\begin{equation}\label{welldefined}
\Gamma_t\cap \partial T_{\pi/4}(C^*) = F_0\cap \partial T_{\pi/4}(C^*) =\overline{\tau}_{1,2k}\cap \partial T_{\pi/4}(C^*).
\end{equation} Then define an embedded (piecewise smooth) surface by setting
\begin{equation}
S_k=\bigcup_{\theta\in [0,2\pi]}\sigma_\theta \circ \rho_{\frac{\theta}{2k}}(\Gamma_\theta).
\end{equation}
Roughly speaking, the cross sections $S_k\cap F_\theta$ first desingularize the $2k$ segments comprising $F_0$ in one of the two possible ways, then rotate the result around $C^*$ so that it comes back with a $\frac{\pi}{k}$ twist desingularizing $F_0$ in the other possible way. 

Finally we define the desired surface with boundary $C$ (in light of \eqref{welldefined}) by:
\begin{equation}
G_k = (\overline{\tau}_{1,k}\cap (\mathbb{S}^3\setminus T_{\pi/4}(C^*)))\cup S_k.
\end{equation}
Since $S_k$ agrees with $\overline{\tau}_{1,2k}$ in a neighborhood of $C$ it follows that \begin{equation}e(S_k)=e(\overline{\tau}_{1,2k})=2k.\end{equation}  
One can readily compute that the genus of $S_k$ is $k$.
\end{proof}

\section{Symmetry of surfaces with boundary a great circle}\label{symmetrysection}

It is a difficult problem to classify the closed embedded minimal surfaces in $\mathbb{S}^3$.  The large six dimensional symmetry group $SO(4)$ of $\mathbb{S}^3$ and equivariant min-max existence theory for each subgroup implies for instance the existence of a plethora of examples.  Indeed, for each subgroup of $SO(4)$ one should obtain a sequence of minimal surfaces realizing the equivariant $p$-widths.

On the other hand, the subgroup $\mathcal{G}\subset SO(4)$ of isometries that preserve the great circle $C$ (set-wise) is much smaller.   Indeed, if $M\in SO(4)$ preserves $C$ then it preserves the $2$-plane $P=\{(z,0)\;|\;z\in\mathbb{C}\}\subset\mathbb{R}^4$.  Since $M$ is orthogonal, it also preserves the orthogonal $2$-plane $P^\perp =\{(0,w)\;|\;w\in\mathbb{C}\}$.  It follows that any such map is block diagonal with respect to $P\oplus P^\perp$.  In other words, 
\begin{equation}
M =
\begin{pmatrix}
M_1 & 0 \\
0 & M_2
\end{pmatrix},
\end{equation}
where $M_1,M_2\in O(2)$ and $\mbox{det}(M_1)\mbox{det}(M_2)=1$.
Thus
\begin{equation}
\mathcal{G}=\mbox{Stab}_{SO(4)}(C) = S(O(2)\times O(2)).
\end{equation}
The finite subgroups of the two dimensional group $\mathcal{G}$ are easy to classify as either $\mathbb{Z}_n$, $\mathbb{D}_n$, $\mathbb{Z}_n\oplus\mathbb{Z}_m$, or $(\mathbb{Z}_n\oplus\mathbb{Z}_m)\rtimes\mathbb{Z}_2$.  As $\mathbb{Z}_{nm}\simeq \mathbb{Z}_n\oplus \mathbb{Z}_m$ when $\gcd(n,m)=1$, one may assume that $n$ and $m$ have a common factor greater than $1$ in the third case and, for similar reasons, in the fourth case.

We show that the latter two cases cannot occur as symmetry groups for any smoothly embedded (let alone minimal) surface with boundary a great circle:
\begin{prop}\label{symmetry}
There is no smoothly embedded surface $\Sigma\subset\mathbb{S}^3$ with boundary $C$ whose isometry group contains a subgroup isomorphic to $\mathbb{Z}_n\oplus\mathbb{Z}_m$ (with $n$ and $m$ sharing a common factor greater than $1$).
\end{prop}
\begin{proof}
Suppose there exists such a surface $\Sigma$.  Consider the restriction map
\begin{equation}
\rho: \mathbb{Z}_n\oplus\mathbb{Z}_m\to \mbox{Isom}(C)\cong O(2).
\end{equation}
The only subgroups of $O(2)$ are cyclic and dihedral groups.  Thus $\rho$ is not an injection.  If $a\in\mbox{Ker}(\rho)$ and $a\neq (e,e)$, then
\begin{equation}
a(z,w) = (z,e^{i\alpha}w)\mbox{ for some } \alpha\in\mathbb{R}.
\end{equation}
Since $a$ fixes $C$ pointwise and $\Sigma$ is embedded, the conormal vector to $\Sigma$ at any point along $C$ is invariant under $a$ as well.  It follows that $\alpha =0$ and $a=(e,e)$, a contradiction.
\end{proof}

It follows from Proposition \ref{symmetry} that the only possible $SO(4)$ isometry groups of a minimal surface with boundary $C$ are cyclic groups $\mathbb{Z}_p$ (subgroups of the isometry group of the $\mu_{p,q}$ surfaces) and dihedral groups $\mathbb{D}_q$ (which occur for the $\tilde{\xi}_{k,m}$ surfaces).  Of course, as described in Remark \ref{coincide}, we expect the two families to coincide and thus that the $\mu_{p,q}$ surfaces have an extra involutive symmetry. It is thus a natural question whether the $\tilde{\xi}_{k,m}$ surfaces and $\overline{\tau}_{1,2}$ are the only embedded minimal equatorial fillings in $\mathbb{S}^3$ (up to ambient isometry).

\section{Questions}\label{questionsection}
The following is a natural analog to the Lawson Conjecture (resolved by S. Brendle \cite{B}):
\begin{conj}
The Lawson M\"obius band is the \emph{unique} embedded minimal M\"obius band with boundary a great circle.  
\end{conj}
The existence of the immersed bands $\bar{\tau}_{1,2k}$ for $k>1$ shows the necessity of embeddedness.  

We have the following non-orientable analog to R. Kusner's  long-standing conjecture for the Lawson surfaces $\xi_{1,g}$ in $\mathbb{S}^3$:
\begin{conj}
For each odd $m\in\mathbb{N}$ with $m\geq 3$, the surface $\tilde{\xi}_{1,m}$ minimizes area and Willmore energy among non-orientable minimal surfaces with boundary a great circle and genus $m+1$.
\end{conj}

It is natural to ask which other simple topological types and Euler numbers can occur:

\begin{question}\label{punctured}
There is no minimal embedded punctured Klein bottle with boundary a great circle. 
\end{question}

Note that by Theorem \ref{admissible} any counterexample to Conjecture \ref{punctured} would necessarily have Euler number $\pm 4$ or $0$.

\begin{conj}\label{2}
The Lawson M\"obius band is the \emph{unique} embedded non-orientable minimal surface with boundary a great circle and with Euler number $\pm 2$.  
\end{conj}

Note that any odd genus in Conjecture \ref{2} is consistent with Theorem \ref{admissible}.  

\begin{conj}\label{0}
A great circle in $\mathbb{S}^3$ bounds no non-orientable embedded minimal surface with Euler number $0$.
\end{conj}
Note that any even genus in Conjecture \ref{0} is consistent with Theorem \ref{admissible}.  
More generally, one can ask whether the lower bound in Theorem \ref{admissible} is saturated:
\begin{conj}
There exists no embedded minimal surface (aside from the Lawson M\"obius band) with boundary a great circle and genus equal to half of its Euler number.
\end{conj}
If $\Gamma$ is a simple closed curve in a $3$-manifold, the \emph{crosscap genus} of $\Gamma$ is the infimal genus of an embedded non-orientable surface $\Sigma$ bounded by $\Gamma$.  
\begin{question}
Is the cross-cap genus of the projection of $C$ to $L(p,q)$ realized by the particular $\tilde{\xi}_{k,m}$ surfaces enumerated in Proposition \ref{crosscap}?  
\end{question}

In higher dimensions, it is natural to ask if \emph{any} minimal equatorial fillings objects exist (note that only in dimension $3$ is a smooth example forced to exist):
\begin{question}
Do there exist non-orientable minimal equatorial fillings in higher dimensional spheres aside from the (singular) iterated suspensions of the ones known in $\mathbb{S}^3$?  Are there any \emph{smoothly embedded} such fillings?
\end{question}

There is a relationship between minimal surfaces in $\mathbb{S}^3$ and self-shrinkers of the mean curvature flow (MCF), see \cite{Bern}, and \cite{KZ} for more discussion.  White \cite{White3pi} showed that there is an embedded self-shrinking M\"obius band with boundary a line in $\mathbb{R}^3$ which arises as a singularity model for certain MCFs with fixed boundary. Thus it is natural in light of this paper to ask:
\begin{question}
Are there embedded higher genus non-orientable self-shrinkers in $\mathbb{R}^3$ with boundary a line?
\end{question}

Finally, it is natural to ask with regard to enumeration:

\begin{question}
How many minimal surfaces are bounded by a \emph{round} circle in $\mathbb{S}^3$?
\end{question}

\begin{question}
How many minimal surfaces are bounded by a curve which is small enough perturbation of a geodesic in $\mathbb{S}^3$?  
\end{question}

L. Caffarelli, B. Hardt and L. Simon \cite{CHS} showed that any non-equatorial minimal hypersurface in $\mathbb{S}^n$ (for $n\geq 3$) is the singularity model for a non-cone minimal hypersurface in $\mathbb{R}^{n+1}\setminus\{0\}$ with an isolated singularity at the origin.  Thus one can ask:
\begin{question}
Do all minimal equatorial fillings arise as singularity models?
\end{question}

\end{document}